\documentclass{scrartcl}
\usepackage{amsmath}
\usepackage{amssymb}
\usepackage{amsthm}
\usepackage[T1]{fontenc}
\usepackage{enumerate}
\usepackage{thm-restate}
\usepackage{hyperref}
\usepackage{todonotes}
\usepackage[mathlines]{lineno}
\usepackage{appendix}
\usepackage{mathtools} 
\usepackage[UKenglish]{babel}
\usepackage{bbm}
\usepackage[utf8]{inputenc}
\usepackage{tikz}
\usepackage{thmtools} 
\usepackage{thm-restate} 
\usepackage{stackrel} 
\usepackage[normalem]{ulem}
\usepackage{color}

\DeclareMathOperator{\lip}{Lip}
\DeclareMathOperator{\id}{id}
\DeclareMathOperator{\conv}{conv}
\newcommand{\norm}[1]{\left\|#1\right\|}
\newcommand{\mc}[1]{\mathcal{#1}}
\DeclareMathOperator{\Span}{Span}

\DeclareMathOperator{\sgn}{\textnormal{sgn}}
\newcommand{\set}[1]{\left\{#1\right\}}
\newcommand{\R}{\mathbb{R}}

\newcommand{\br}[1]{\left(#1\right)}
\newcommand{\N}{\mathbb{N}}
\newcommand{\Q}{\mathbb{Q}}
\newcommand{\wt}[1]{\widetilde{#1}}

\newtheorem{proposition}{Proposition}[section]
\newtheorem{corollary}[proposition]{Corollary}

\newtheorem{theorem}[proposition]{Theorem}
\newtheorem{lemma}[proposition]{Lemma}
\theoremstyle{definition}
\newtheorem{definition}[proposition]{Definition}
\theoremstyle{remark}
\newtheorem{remark}[proposition]{Remark}
\newtheorem*{remark*}{Remark}
\newtheorem{example}[proposition]{Example}

\title{On extremal nonexpansive mappings}
\author{Christian Bargetz \and Michael Dymond \and Katriin Pirk}

\begin{document}
\maketitle
\begin{abstract}
  \noindent\textbf{\textsf{Abstract.}} We study the extremality of nonexpansive mappings on a nonempty bounded closed and convex subset of a normed space (therein specific Banach spaces). We show that surjective isometries are extremal in this sense for many Banach spaces, including Banach spaces with the Radon-Nikodym property and all $C(K)$-spaces for compact Hausdorff $K.$ We also conclude that the typical, in the sense of Baire category, nonexpansive mapping is close to being extremal. 
  \vskip2mm
  \noindent\textbf{\textsf{Keywords.}} Banach space, generic property, nonexpansive mapping.
  \vskip2mm
  \noindent\textbf{\textsf{Mathematics Subject Classifications.}} 46B25, 47H09, 54E52

\end{abstract}
\section{Introduction}
While in Hilbert spaces it is easy to see that every isometry is an affine mapping, in Banach spaces the situation is more complicated. The classical Mazur-Ulam Theorem, see~\cite{MazurUlam}, states that every \emph{surjective} isometry between normed spaces is an affine map. This result has been extended to surjective isometries of convex bodies in normed spaces by P.~Mankiewicz in Theorem~5 of~\cite{Mankiewicz}.

Given a bounded, closed and convex subset $C$ of a Banach space $X$, the space $\mathcal{M}$ of nonexpansive self-mappings of $C$ is usually equipped with the metric of uniform convergence which turns it into a complete metric space. An alternative point of view is to consider this space as a convex subset of the space $\mathcal{CB}(C,X)$ of bounded continuous mappings from $C$ to $X$. 

Since $C$ is bounded, $\mathcal{M}$ is also a bounded subset of $\mathcal{CB}(C,X)$. In many cases, e.g. when the Krein-Milman Theorem can be applied or more general for a space with the Krein-Milman property, convex sets can be described by their extreme points. One of the goals of the current article is to shed some light on the set of extreme points of $\mathcal{M}$ and to investigate how much information on the whole set is contained in the set of extreme points. In particular, we ask the question of whether surjective isometries are extreme points and provide the following answer in a wide class of Banach spaces.

\begin{theorem}\label{thm:collect}
  Let $X$ be a Banach space satisfying at least one of the following conditions:
  \begin{enumerate}[(i)]
  \item\label{almexp} The closed unit ball $\mathbb{B}_{X}$ is the closed convex hull of its almost exposed points.
  \item\label{CK} $X=C(K)$ for some compact Hausdorff topological space $K$.
  \item\label{czero} $X=c_0$.
  \end{enumerate}
  Then every surjective isometry $\mathbb{B}_{X}\to \mathbb{B}_{X}$ is extremal in the space of non-expansive mappings $\mathbb{B}_{X}\to \mathbb{B}_{X}$.
\end{theorem}

In condition~\eqref{almexp} of Theorem~\ref{thm:collect} the notion of an almost exposed point is weaker than the standard notion of an exposed point, although we do not know whether the two notions may coincide. We postpone further details and discussion on this until Definition~\ref{def:almexp}. For now, we note that all Banach spaces with the Radon-Nikodym property satisfy the condition~\eqref{almexp} of Theorem~\ref{thm:collect}. Consequently, Theorem~\ref{thm:collect} captures all $\ell_{p}$ and $L^{p}[0,1]$ spaces with $1< p< \infty$. We also capture $\ell_{1}$ (via \eqref{almexp}) and $c_{0}$, as well as all $C(K)$ spaces. Since $\ell_\infty$ and $L_\infty[0,1]$ are isometrically isomorphic to $C(K)$-spaces, the theorem applies to them as well. Hence, arguably the only `classical' Banach space which evades Theorem~\ref{thm:collect} is $L^1[0,1]$. We do not know whether surjective isometries of the closed unit ball of $L^1[0,1]$ are extremal.

In the linear case this question is well-studied, i.e. the question of finding description of the extreme points of the unit ball of the space $L(X,X)$ of bounded linear operators, called extreme contractions operators or extreme contractions, has been considered for a number of Banach spaces~$X$. Let us briefly mention a number of these results. In~\cite{BLP} R.~M.~Blumenthal, J.~Lindenstrauss and R.~R.~Phelps provide the following characterisation of extreme contraction operators between $C(K)$-spaces: A norm-one operator $T\colon C(K_1)\to C(K_2)$ is an extreme point of the unit ball if and only if there is a continuous function $\psi\colon K_2\to K_1$ and $\varphi\in C(K_2)$ with $|\varphi|=1$ such that $(Tf)(t) = \varphi(t) f(\psi(t))$. Since every surjective linear isometry is of this form where $\psi$ is even a homeomorphism, in particular surjective linear isometries are extreme contraction operators. In~\cite{Kim1976:ExtremeContraction}, the extreme contraction operators on~$\ell_\infty$ are characterised by order theoretic and algebra properties. For more results on extreme contraction operators and in particular on the question of when all of them are isometries, we refer the interested reader to~\cite{NPN,Reich,Sain,SPM} and the references therein.

In the context of fixed point theory, the subsets of strict contractions within spaces of nonexpansive mappings are of clear significance. Banach's Fixed Point Theorem guarantees the existence of a unique fixed point of each strict contraction between complete metric spaces. Existence of a fixed point is an extremely useful property and has numerous applications in diverse areas of mathematics, for example in iterated function systems in fractal geometry and in solving partial differential equations. Therefore, it is a matter of high interest to determine to what extent the fixed point property extends beyond just the set of strict contractions, most obviously to the natural complete metric space of mappings in which this set lives, namely spaces of nonexpansive mappings, i.e. mappings with Lipschitz constant at most one. For more details on the fixed point theory of nonexpansive mappings, we refer the interested reader to~\cite{GR:Nonexpansive} and~\cite{FixedPointTheory}.

Whilst it is clear that mappings between two complete metric spaces with Lipschitz constant one need not have a fixed point (consider, for example a translation of the real line), it has been shown that in an important setting the vast majority of nonexpansive mappings retain the fixed point property. 
In finite dimensional spaces, Brouwer's Fixed Point Theorem provides a fixed point for every continuous self mapping of any given bounded, closed and convex set. This incredibly strong result is not valid in infinite dimensional spaces. However, inside the space of nonexpansive self mappings of a bounded, closed and convex subset of a general (possibly infinite dimensional) Banach space, results of de Blasi and Myjak, see~\cite{DM1976Convergence, DM1989Porosity}, guarantee the existence of a fixed point for every nonexpansive mapping except for those forming a negligible subset of the space of nonexpansive mappings. 

Of course, the previous statement only has power beyond the scope of Banach's Fixed Point Theorem if the subset of strict contractions forms a negligible subset of the space of nonexpansive mappings and the latter is indeed verified by the first two named authors in \cite{bargetz2016sigma}. Reich and Zaslavski, see~\cite{RZ2001NoncontractiveMappings}, showed that in this setting most nonexpansive mappings satisfy the assumptions of Rakotch's fixed point theorem, see~\cite{Rak1962Contractive}. When combined these results show that the phenomenon exhibited by de Blasi and Myjak can be explained by Rakotch's fixed point theorem but not by the one of Banach. It turns out that for these results to hold true, boundedness of $C$ is essential. If $C$ is an unbounded closed and convex subset of a Hilbert space, F.~Strobin showed in~\cite{Strobin} that the set of Rakotch contractions is in fact a meagre subset of the space of nonexpansive self-mappings. This result was later generalised to Banach spaces by Reich, Thimm and the first author in~\cite{BRT} and in a more general setting by Ravasini in~\cite{Ravasini}. On the other hand Reich and Zaslavski showed in~\cite{RZ-unbounded} that also in the unbounded case the property of having a fixed point is still typical. 

The present article is motivated in part by the observation that any nontrivial convex combination of elements in $\mathcal{M}$, that is any mapping of the form
\begin{equation}\label{eq:nontrivial_convex}
	(1-\lambda)g+\lambda h, \qquad \text{with $\lambda\in (0,1)$ and $g\neq h$,}
\end{equation}
is a sum of two strict contractions. Since strict contractions form a negligible subset of $\mathcal{M}$, we consider the question of whether the collection of all non-trivial convex combinations in $\mc{M}$ is also small in some sense. We show that for nearly all mappings $f\in \mc{M}$ there is only a tiny subset of $\mc{M}$ of mappings $g$ which may participate in a convex combination of the form \eqref{eq:nontrivial_convex} giving $f$.
\begin{restatable}{theorem}{Pf}\label{thm:Pf}
  Let $X$ be normed space, $C\subseteq X$ be a non-empty, non-singleton, closed, bounded and convex set and $\mc{M}=\mc{M}(C)$ be the space of nonexpansive mappings $C\to C$ equipped with the metric $d_{\infty}(f,g)=\sup_{x\in C}\norm{f(x)-g(x)}$. For each $f\in \mc{M}$ let 
  \begin{equation*}
    P_{f}:=\set{g\in\mc{M}\colon \exists \lambda\in [0,1),\, \exists h\in \mc{M}\,\text{s.t.}\,f=(1-\lambda)g+\lambda h}
  \end{equation*}
  and let $\mc{N}$ denote the set of strict contractions $C\to C$. Then the following statements hold:
  \begin{enumerate}[(i)]
  \item\label{M-Pf_convex} For every $f\in\mc{M}$, $P_{f}$ is $F_{\sigma}$, $\mc{M}\setminus P_{f}$ is convex and there exists an affine subspace $A_{f}$ of the space of continuous mappings $C\to X$ such that $P_{f}=\mc{M}\cap A_{f}$.    
  \item\label{lipf=1} For every $f\in\mc{M}\setminus \mc{N}$, i.e. every $f\in\mc{M}$ with $\lip(f)=1$, we have that $P_{f}$ is a $\sigma$-upper porous subset of $\mc{M}$. In fact $P_{f}$ is a countable union of closed, upper porous subsets of $\mc{M}$.
  \end{enumerate}
\end{restatable}

The rest of the paper is structured as follows: Section \ref{sec: Preliminaries} builds up the necessary preliminary knowledge: explains both the notation and the notions at the heart of this article. In Section \ref{sec:surj_isom} we present some general observations regarding the extremality of surjective isometries on a bounded, closed and convex subset, e.g. we prove Theorem~\ref{thm:collect}, (\ref{almexp}), i.e. show for Banach spaces that under some specific conditions, that include, for example, Banach spaces with the Radon-Nikodym property, surjective isometries are extremal. Thereafter, in Section \ref{sec:CK_c0}, we focus on Theorem~\ref{thm:collect}, (\ref{CK}) and (\ref{czero}) to show that under the additional assumption that the bounded, closed and convex set is the unit ball, surjective isometries are extremal in $C(K)$-spaces for compact Hausdorff $K$ and $c_0.$ In Section \ref{sec: extr of typ nonexp map} we investigate, in the sense of Baire Category Theorem, 
how often non-trivial convex combinations occur in the space of nonexpansive mappings $\mathcal{M}$ discussed above and as a conclusion prove Theorem \ref{thm:Pf}.
\section{Preliminaries}\label{sec: Preliminaries}
In the following we give a quick overview of the notation, key-notions and handy tools used in this article. In Sections \ref{sec:surj_isom} and \ref{sec:CK_c0} we focus on Banach spaces, in Section \ref{sec: extr of typ nonexp map} the approach is broader and we deal with normed spaces. We use standard notation, e.g. for a ball in metric space $M$ with center $a\in M$ and radius $r>0$, we write $B_M(a,r)$. We use $S_X$ to denote the unit sphere of a Banach space $X,$ and $\mathbb{B}_{X}$ to clearly distinguish the closed unit ball of a Banach space $X.$ For a vector space $Y$ we denote the linear span of its $n\in\N$ elements $y_1,\dots,y_n$ by $\Span\{y_1,\dots,y_n\}.$

Now let us introduce some of the core notions of this paper. Let $X$ be a normed space and $C\subseteq X$ be a non-empty, non-singleton, closed, bounded and convex subset. 
We call a mapping $F\colon C\to C$ \emph{nonexpansive} if for every $x,y\in X$
\begin{equation}
  \|F(x)-F(y)\|\leq \|x-y\|.
\end{equation}
In other words, a nonexpansive mapping $F$ is a Lipschitz mapping with Lipschitz constant 1 or less, i.e. $\lip{(F)}\leq 1$. We say that a nonexpansive mapping is a \emph{strict contraction} if its Lipschitz constant is strictly less than one.

Throughout the paper we denote by $\mathcal{M}=\mathcal{M}(C)$ the space of all nonexpansive mappings on $C.$ If $C$ is bounded, $\mathcal{M}$ is a complete metric space when equipped with the metric of uniform convergence. We denote the subset of all strict contractions in~$\mathcal{M}$ by~$\mathcal{N}.$

As is evident from the title of this paper, the main objects of interest in this article are \emph{extremal nonexpansive mappings}. Let us now clarify what we mean by extremality in this setting.
\begin{definition}
  We call a mapping $f\in \mathcal{M}$ \emph{extremal} if it does not have a representation as a nontrivial convex combination of two elements of $\mathcal{M},$ i.e. $f\in \mathcal{M}$ is said to be \emph{extremal} if $(1-\lambda) g+\lambda h\neq f$ for every $\lambda\in (0,1)$ and every $g,h\in \mathcal{M}\setminus\set{f}$. 

  Given a subset $\mc{F}$ of $\mc{M}$ we say that $f\in\mc{F}$ is extremal among mappings from $\mc{F}$ if $(1-\lambda) g+\lambda h\neq f$ for every $\lambda\in (0,1)$ and $g,h\in\mc{F}\setminus\set{f}$.
\end{definition}

The Sections \ref{sec:surj_isom} and \ref{sec:CK_c0} will be devoted to the proof of Theorem~\ref{thm:collect}. We show that surjective isometries of the closed unit ball are extremal in nearly all classical Banach space settings. Recall that a mapping $f\colon C \to C$ is an \emph{isometry} if it is distance preserving, i.e. for every $x,y\in C$ we have
\begin{equation}\label{eq: isometry}
  \|f(x)-f(y)\|=\|x-y\|.
\end{equation}
We say that $f\in \mathcal{M}$ is a \emph{surjective isometry} if $f\colon C\to C$ is an isometry and is onto. Although we do not require isometries to be linear by default, the following lemma (see Lemma 1 and Theorem 5 in \cite{Mankiewicz}) specifies conditions under which we naturally end up with linear surjective isometries. It is useful to note that our results can be divided into two. Some hold for isometries on any set $C$ and some look at the more restricted case where $C:=\mathbb{B}_X.$ Linearity of a surjective isometry in the latter case, where $C:=\mathbb{B}_{X}$, is assured:
\begin{lemma}\label{lem: Mankiewicz}
  Let $X$ be a Banach space and $T\colon \mathbb{B}_X\to \mathbb{B}_X$ a surjective isometry. Then $T$ is the restriction of a linear mapping $X\to X$.
\end{lemma}
\begin{proof}
  Using Lemma 1 in \cite{Mankiewicz} we see that $T$ fixes $0\in \mathbb{B}_X.$ This means that $T(\mathbb{B}_X)$ is a convex body. Now we apply Theorem 5 in \cite{Mankiewicz} and get that $T$ extends uniquely to an affine isometry from $X$ to $X.$ Therefore, $T$ is linear. 
\end{proof}

In Section \ref{sec: extr of typ nonexp map} we prove Theorem~\ref{thm:Pf}, which shows that non-trivial convex combinations are rare in a certain sense among nonexpansive mappings. There are different ways to describe how rare some type of elements are. In this paper we use the notion of \emph{$\sigma$-upper porosity}, due to Zajíček~\cite[Definition~2.1]{Zajicek2005}. The following definition is based on an equivalent formulation of Zajíček's definition, given by \cite[Lemma~2.2]{dymond2023porosity}
\begin{definition}
  Let $(M,d)$ be a metric space without isolated points, 
  \begin{enumerate}[(i)]
  \item We say that a set $P\subseteq M$ is upper porous at a point $q\in M$ if there exists $\alpha\in (0,1)$ such that for every $\varepsilon>0$ there exists $\wt{q}\in M$ such that $0<d(q,\wt{q})\leq \varepsilon$ and $B_{M}(\wt{q},\alpha d(q,\wt{q}))\cap P=\emptyset$. 
  \item We say that a set $P\subseteq M$ is lower porous at a point $q\in M$ if there exists $\alpha\in (0,1)$ and $\varepsilon_{0}>0$ such that for every $\varepsilon\in(0,\varepsilon_0)$ there exists $\wt{q}\in M$ such that $d(q,\wt{q})\leq \varepsilon$ and $B_{M}(\wt{q},\alpha\varepsilon)\cap P=\emptyset$.
  \item We call a set $P\subseteq M$ upper porous (respectiviely lower porous) if it is upper porous (respectively lower porous) at every point $q\in P$.
  \item A set $E\subseteq M$ is called $\sigma$-upper porous (respectively $\sigma$-lower porous) if it is a countable union of upper (respectively lower) porous subsets of $M$.
  \end{enumerate}

\end{definition}
Note that upper porosity is a weaker notion than lower porosity. Moreover, both notions are stronger than the notion of nowhere dense. Therefore, in a metric space the classes of $\sigma$-upper and of $\sigma$-lower porous sets lie within the class of meagre subsets. Thus, $\sigma$-upper and $\sigma$-lower porous subsets of a complete metric space should be thought of as even more negligible than meagre subsets or subsets of the first Baire Category, which are small due the Baire Category Theorem.
\section{General observations on the extremality of surjective isometries}\label{sec:surj_isom}

The main goal of this section is to prove Theorem~\ref{thm:collect}, (\ref{almexp}).

Before we investigate in which settings surjective isometries are extremal, let us observe that this property is equivalent to a seemingly weaker one.

\begin{lemma}\label{lemma:id_reduction}
  Let $X$ be a Banach space and $C\subset X$ be a closed and convex subset with nonempty interior. Let $\mc{M}$ denote the set of nonexpansive self mappings of $C$. Then every surjective isometry $C\to C$ is extremal in $\mc{M}$ if and only if the identity $\id\colon C\to C$ is extremal in $\mc{M}$.
\end{lemma}

\begin{proof}
  The `only if' part is trivial. We prove the `if' statement: Suppose that $\id\in\mc{M}$ is extremal, let $f\in\mc{M}$ be a surjective isometry and suppose there exist $\lambda\in (0,1)$ and $g,h\in \mc{M}$ such that $f=(1-\lambda)g+\lambda h$. Our task is to show that $g=f$. By Theorem~5 in~\cite{Mankiewicz} we have that $f$ (and therefore also $f^{-1}$) is the restriction of an affine mapping $X\to X$. Therefore, applying $f^{-1}$ to the above convex representation of $f$, we get $\id=(1-\lambda)f^{-1}\circ g+\lambda f^{-1}\circ h$ and note that $f^{-1}\circ g,f^{-1}\circ h\in\mc{M}$. Since $\id$ is extremal in $\mc{M}$, we conclude that $f^{-1}\circ g=\id$ and therefore, that $g=f$.
\end{proof}

As a next general observation, we obtain that the identity is extremal among surjective nonexpansive mappings.

\begin{proposition}\label{prop:LinearCase}
  Let $X$ be a Banach space and $C\subset X$ be a bounded, closed and convex subset with nonempty interior. The identity is extremal among surjective isometries of~$C$.
\end{proposition}

For the proof we need the following lemmas.

\begin{lemma}\label{lem:LinearDual}
  Let $X$ be a Banach space and $E$ be the set of all extreme points of the unit ball $B_{X}$. Assume that there are two linear operators $f,g\colon X\to X$ of norm at most one. Let $\lambda\in (0,1)$ such that $\mathrm{id} =  (1-\lambda) f + \lambda g$. Then, $f(x)=g(x)$ for all $x\in \operatorname{conv}(E)$.
\end{lemma}

\begin{proof}
  First note that for $y\in E$ we have $y = \lambda f(y) + (1-\lambda)g(y)$ which by extremality of $y$ means that $y=f(y)=g(y)$. Given $x\in \operatorname{conv}(E)$, there exist $m\in\N$, $y_{1},\ldots,y_{m}\in E$ and $\mu_{1},\ldots,\mu_{n}\in\R$ such that $x = \sum_{i=1}^{m} \mu_i y_i$
  and hence
  \[
    f(x)=\sum_{i=1}^{m} \mu_i f(y_i) = \sum_{i=1}^{m} \mu_iy_i = x
  \]
  and similar for $g(x)$ which finishes the proof.
\end{proof}

\begin{lemma}\label{lem: surj isom extr among lin nonexp}
  Let $X$ be a Banach space. The linear surjective isometries $X\to X$ are extremal among the linear nonexpansive mappings $X\to X$.
\end{lemma}

\begin{proof}
  Similarly to the proof of Lemma~\ref{lemma:id_reduction} we may conclude that extremality of the identity among linear nonexpansive self-mappings of $X$ implies that surjective isometries are extremal. Let $f,g\colon X\to X$ be nonexpansive linear mappings. Their adjoints $f^*$ and $g^*$ are linear nonexpansive mappings. By the Krein-Milman theorem $\mathbb{B}_{X^*}$ is the weak$^*$ closed convex hull of the set $E$ of its extreme points. By Lemma~\ref{lem:LinearDual} on $E$ the mappings $f^*$ and $g^*$ agree with the identity. Since they are adjoints of bounded linear mappings, they are weak$^*$-weak$^*$ continuous. Hence they each have a unique extension to the weak$^*$ closure of $\conv E$ and therefore $f^*=g^*=\operatorname{id}$ which in turn implies that $f=g=\operatorname{id}$.
\end{proof}

\begin{proof}[Proof of Proposition~\ref{prop:LinearCase}]
  Assume that there are surjective isometries $f,g\colon C\to C$ and $\lambda\in (0,1)$ such that $\operatorname{id} = (1-\lambda) f + \lambda g$. Since $f$ and $g$ are surjective isometries, by Theorem~5 in~\cite{Mankiewicz} they extend to affine and bijective isometries $X\to X$ which we again denote by $f$ and $g$. We denote by $f_0$ and $g_0$ their linear parts which are obviously surjective linear isometries. Now we may apply Lemma~\ref{lem: surj isom extr among lin nonexp} to conclude that that $f_0=g_0=\operatorname{id}$.
\end{proof}

We can gain some more information on mappings which are able to form a non-trivial convex combination which is an isometry.

\begin{proposition}
  Let $X$ be a Banach space and $C\subset X$ be a bounded, closed and convex set. Moreover let $f,g\in\mc{M}$ and $\lambda\in (0,1)$ such that $(1-\lambda) f +\lambda g$ is an isometry. Then $f$ and $g$ are isometries.
\end{proposition}

\begin{proof}
  The statement follows from
  \begin{align*}
    \|x-y\| &= \|(1-\lambda) f(x) +\lambda g(x) - (1-\lambda) f(y) - \lambda g(y)\|\\ & \leq (1-\lambda) \|f(x)-f(y)\| + \lambda \|g(x)-g(y)\|\\ &\leq \|x-y\|,
  \end{align*}
  for all $x,y\in C$.
\end{proof}

Recall that for a subset $C\subset X$ of a Banach space a point $x\in C$ is called \emph{exposed} if there is a functional $\varphi\in X^*$ with
\[
  \varphi(x) = \sup_{z\in C} \varphi(z) =:s \qquad \text{and}\qquad \varphi^{-1}(s)\cap C = \{x\},
\]
see e.g. Definition~7.10 in~\cite[p.~336]{CzechBible}.
For convex sets $C$ and $x\in C$ we might rephrase this to require that there is a hyperplane supporting $C$ in $x$ which intersects $C$ precisely in $x$. We use the following weaker notion.

\begin{definition}\label{def:almexp}
  Let $X$ be a Banach space and $C\subset X$ a closed and convex set. A point $x\in C$ is called \emph{almost exposed} if the intersection of all hyperplanes supporting $C$ in $x$ is the singleton $x$.
\end{definition}

\begin{remark}
  Since every exposed point is almost exposed it seems natural to ask whether the converse is also true, i.e. if every almost exposed point is already an exposed point. We do not know the answer to this question.
\end{remark}

\begin{proposition}
  Let $X$ be a Banach space and $C\subset X$ be a strictly convex closed set with nonempty interior. Then every element of the boundary of $C$ is exposed and hence almost exposed.
\end{proposition}

\begin{proof}
  Let $x\in \partial C$. By the Hahn-Banach theorem there is a functional $\varphi\in X^*$ with $1=\varphi(x) > \varphi(z)$ for all $z$ in the interior of $C$. Since both $C$ and the hyperplane $H:=\{x\in X\colon\varphi(x)=1\}$ are convex and the boundary of $C$ does not contain a line segment, we conclude that $H\cap C=\{x\}$.
\end{proof}

Recall that given a set $C\subset X$ a point $x\in C$ is called \emph{strongly exposed} if there is a functional $x^*\in X^*$ with $\langle x^*, x\rangle = \sup \{\langle x^*, z\rangle \colon z\in C\}$ and $x_n\to x $ for all sequences in~$C$ with $\lim_{n\to\infty} \langle x^*, x_n\rangle = \langle x^*, x\rangle$. Note that every strongly exposed point of $C$ is an exposed point of $C$ and hence in particular an almost exposed point of $C$.

\begin{proposition}
  Let $X$ be a Banach space with the Radon-Nikodym property. Then every bounded closed convex set $C\subset X$ is the closed convex hull of its almost exposed points.
\end{proposition}

\begin{proof}
  This is a direct consequence of the fact that every bounded closed convex subset of
  a Banach space with the Radon-Nikodym property is the closed convex hull of its strongly exposed points, see \cite{Phelps}.
\end{proof}

In Lemma~\ref{lem:LinearDual}, and hence in Proposition~\ref{prop:LinearCase}, we are able to show that on sets which are the closed convex hull of its extreme points, the identity is extremal among \emph{linear} nonexpansive mappings. For a corresponding nonlinear result, we need the stronger condition that the underlying convex set is the closed convex hull of its almost exposed points.

\begin{theorem}[Theorem~\ref{thm:collect}, (\ref{almexp})]\label{thm:GeneralSurjIsom}
  Let $X$ be a Banach space with the property that its unit ball $\mathbb{B}_X$ is the closed convex hull of its almost exposed points. Then surjective isometries are extremal in the space of nonexpansive self-mappings of the unit ball.
\end{theorem}
\begin{proof}
  It suffices, in view of Lemma~\ref{lemma:id_reduction}, to prove that the identity, $\id\in \mc{M}$, is extremal. Suppose that there exist $\lambda\in (0,1)$ and $g,h\in \mc{M}$ such that $\id=(1-\lambda)g+\lambda h$. We complete the proof by showing that $g=\id$.

  We denote by $E$ the set of almost exposed points of $\mathbb{B}_X$. For $e\in E$, $x\in \mathbb{B}_{X}$, $t\in (0,1-\norm{x})$ and every functional $\varphi\in X^*$ with $\varphi(e)=\sup\{\varphi(z)\colon z\in B_X\}$ we have
  \begin{multline*}
    \varphi(e) = \varphi\br{\frac{\id(x+te)-\id(x)}{t}}\\ =(1-\lambda) \varphi\br{\frac{g(x+te)-g(x)}{t}} + \lambda\varphi\br{\frac{h(x+te)-h(x)}{t}}\leq \varphi(e)
  \end{multline*}
  We conclude that
  \begin{equation*}
    \frac{g(x+te)-g(x)}{t}\in \mathbb{B}_{X} \cap \set{y\in X\colon \varphi(y) =\varphi(e)}
  \end{equation*}
  Since the above is true for every functional $\varphi$ defining a hyperplane which supports $\mathbb{B}_X$ in $e$ and $e$ is almost exposed, we have shown that $\frac{g(x+te)-g(x)}{t}=e$. We have established this for an arbitrary $e\in E$, $x\in \mathbb{B}_{X}$ and $t\in (0,1-\norm{x})$. Therefore, in summary, we have proved that
  \begin{equation*}
    g(x+te)=g(x)+te\qquad\text{for every $e\in E$, $x\in \mathbb{B}_{X}$ and $t\in (0,1-\norm{x})$.}
  \end{equation*}
  Therefore, by induction, we get that
  \begin{multline*}
    g\br{\sum_{j=1}^{n}\alpha_{j}e_{j}}=g(0)+\sum_{j=1}^{n}\alpha_{j}e_{j} \qquad \text{for every $n\in \mathbb{N}$, $e_{1},\ldots,e_{n}\in E$}\\\text{and $\alpha_{1},\ldots,\alpha_{n} \geq 0$ such that $\sum_{j=1}^{n}\alpha_{j}=1$.}
  \end{multline*}
  This shows that on the convex hull of $E$, $g$ coincides with the translation by $g(0)$. By continuity of $g$ this is also true on the closed convex hull of $E$ which is $\mathbb{B}_{X}$. Since the only translation which maps the unit ball to itself is the identity, we conclude that $g=\id$.
\end{proof}

Since the unit ball of a strictly convex Banach space is an expand-contract plastic space, see~\cite{CKZ}, i.e. on the unit ball of a strictly convex space every bijective nonexpansive mapping is already an isometry, we obtain the following corollary.

\begin{corollary}
  Let $X$ be a strictly convex Banach space with the Radon-Nikodym property. The bijective nonexpansive mappings are extremal in the space of nonexpansive self-mappings of the unit ball~$\mathbb{B}_X$.
\end{corollary}

\begin{remark}
  Since the unit ball of~$\ell_1$ also is an expand-contract plastic space, see~\cite{KZ}, and all $\ell_p$-spaces, $1\leq p<\infty$ have the Radon-Nikodym property, we obtain that on the unit ball $\mathbb{B}_{\ell_p}$, $1\leq p <\infty$, bijective nonexpansive mappings are extremal.
\end{remark}

\section{Extremality of surjective isometries on the unit ball of the space $C(K)$ or $c_0$}\label{sec:CK_c0}
This section is in general dedicated to showing the extremality of surjective isometries on the unit ball of the function space $C(K)$ for compact Hausdorff $K$ and the unit ball of the sequence space $c_0,$ i.e. proving Theorem~\ref{thm:collect}, (\ref{CK}) and (\ref{czero}), first of which we restate as Theorem \ref{thm: C(K)-result} in this section. We start by obtaining the result for surjective isometries $\mathbb{B}_{C(K)}\to \mathbb{B}_{C(K)}$ and then present the related results in the setting of $\mathbb{B}_{c_0}.$ At the end of Section \ref{sec:CK_c0} we present an equivalent condition for extremality of linear mappings on $\mathbb{B}_{c_0}$.

Our first aim is to show that surjective isometries on the unit ball of $C(K)$ are extremal (see Theorem \ref{thm: C(K)-result}). Thanks to Lemma~\ref{lemma:id_reduction}, it is sufficient to check this only for the identity mapping. The proof of Theorem \ref{thm: C(K)-result} relies on two important observations that will be stated as propositions (see Proposition \ref{prop: F,G as id+-1} and Proposition \ref{prop: G is sur isom all K}). The first observation shows that if the identity mapping has a representation as a nontrivial convex combination of two nonexpansive mappings, then these mappings mimic the identity mapping on some specific arguments.

\begin{proposition}\label{prop: F,G as id+-1}
Let $\mc{M}$ denote the set of nonexpansive mappings $\mathbb{B}_{C(K)}\to \mathbb{B}_{C(K)}$. If there exist $F,G\in \mathcal{M}$ and $\lambda\in (0,1)$ such that $\id=(1-\lambda) F+ \lambda G,$ then both $F$ and $G$ have to be such that for every $f\in S_{C(K)}$ and $x\in K$ with $f(x)=\pm 1$ we have
\begin{equation}\label{eq: F(f)(x)=G(f)(x)=id}
F(f)(x)=G(f)(x)=\id(f)(x)=\pm 1.
\end{equation}
        
\end{proposition}
\begin{proof}
We show the equality (\ref{eq: F(f)(x)=G(f)(x)=id}) for $F$, the proof for $G$ is essentially the same. Fix $f\in S_{C(K)}$ and $x\in K$ such that $|f(x)|=1.$ Assume at first that $f(x)=1.$ Then $\id(f)(x)= 1.$ Note that if $F(f)(x)\neq 1,$ then we have that $F(f)(x)<1$ and therefore
    $$1=(1-\lambda) F(f)(x) + \lambda G(f)(x)< (1-\lambda)+\lambda=1,$$
    which is a contradiction.
Now assume that $f(x)=-1.$ If $F(f)(x)\neq \id(f)(x)=-1$ then $F(f)(x)>-1$ and therefore
    $$-1=(1-\lambda) F(f)(x) + \lambda G(f)(x)>-(1-\lambda)- \lambda=-1,$$
which is a contradiction. Hence, (\ref{eq: F(f)(x)=G(f)(x)=id}) holds for every $f\in S_{C(K)}$ and $x\in K$ with $f(x)=\pm 1.$
\end{proof}
In the following we intend to provide a converse to Propostion~\ref{prop: F,G as id+-1}. We namely show, in Proposition~\ref{prop: G is sur isom all K}, that once a nonexpansive mapping $\mathbb{B}_{C(K)}\to \mathbb{B}_{C(K)}$ satisfies (\ref{eq: F(f)(x)=G(f)(x)=id}) for every $f\in S_{C(K)}$ and $x\in K$ with $f(x)=\pm 1$ (as in Proposition \ref{prop: F,G as id+-1}), it has to be equal to the identity mapping everywhere. The proof of this key-observation makes use of the following lemma that allows us to choose for any element in $\mathbb{B}_{C(K)}$ a suitable norm-1 element that agrees with the initial function over most of the domain but in a specific place guarantees we can apply Proposition \ref{prop: F,G as id+-1}.

\begin{lemma}\label{lem: C(K) g+g-}
        Let $K$ be a compact Hausdorff topological space. For every $f\in \mathbb{B}_{C(K)},$ every $x_0\in K$ with $|f(x_0)|\neq 1,$ and every $0<\gamma<\min\{|-1-f(x_0)|,|1-f(x_0)|\}$ there is $U_{x_0}$ a neighbourhood of $x_0$ for which we can define $g_+,g_-\in \mathbb{B}_{C(K)}$ such that
    \begin{enumerate}[(i)]
        \item\label{i} $g_+(x_0)=1$ and $g_-(x_0)=-1,$
        \item\label{ii} $g_+(x)=f(x)=g_-(x)$ for every $x\in K\backslash U_{x_0},$
        \item\label{iii} $\|f-g_+\|\leq 1-f(x_0)+\gamma,$
        \item\label{iv} $\|f-g_-\|\leq 1+f(x_0)+\gamma.$
    \end{enumerate}
\end{lemma}
The proof of Lemma \ref{lem: C(K) g+g-} makes use of the well-known Urysohn lemma (see e.g.  Theorem 3.1 in \cite{munkres2000topology}), stated here (without proof) for the convenience of the reader.

\begin{lemma}[Urysohn lemma]\label{lemma:Urysohn}
    Let $X$ be a normal space, let $A$ and $B$ be disjoint closed subsets of $X.$ Let $[a,b]$ be a closed interval in the real line. Then there exists a continuous map
    \[r\colon X\to [a,b]\]
    such that $r(x)=a$ for every $x\in A$ and $r(x)=b$ for every $x\in B.$
\end{lemma}
\begin{proof}[Proof of Lemma \ref{lem: C(K) g+g-}]
    Fix $f\in \mathbb{B}_{C(K)},$ $x_0\in K.$ Fix $0<\gamma<\min\{|-1-f(x_0)|,|1-f(x_0)|\}.$ Since $f$ is continuous, we get an open $x_0$ neighbourhood $U_{x_0}:=f^{-1}((f(x_0)-\gamma,f(x_0)+\gamma)).$ Define continuous functions $g_+\colon K\to [-1,1]$ and $g_-\colon K\to [-1,1]$ in the following way:
    \begin{itemize}
        \item [$(+)$]$g_+(x)=f(x)+r(x)(1-f(x))$,
        \item [$(-)$]$g_-(x)=f(x)-r(x)(1+f(x))$,
    \end{itemize}
    where $r\colon K\to [0,1]$ is the continuous map obtained by the Urysohn lemma (Lemma~\ref{lemma:Urysohn}) with $A=K\backslash U_{x_0},$ $B=\{x_0\},$ $a=0,$ $b=1,$ i.e.:
    \begin{equation}\label{eq: r(x)}
        r(x)=\begin{cases}
            0,\quad\text{if}\quad x\in K\backslash U_{x_0},\\
            1,\quad\text{if}\quad x=x_0.
        \end{cases}
    \end{equation}
    
    It is easy to check that the values of both $g_+$ and $g_-$ take values in the interval $[-1,1]$ and are continuous, so $g_+,g_-\in \mathbb{B}_{C(K)}.$ Moreover, $g_+(x)\in [f(x_0)-\gamma,1]$ and $g_-(x)\in [-1,f(x_0)+\gamma]$ for all $x\in U_{x_{0}}$.
    
    Now we will show that conditions \eqref{i}--\eqref{iv} are satisfied.
    Condition \eqref{i} is true, as
    $g_+(x_0)=f(x_0)+r(x_0)(1-f(x_0))=f(x_0)+1-f(x_0)=1,$ and $g_-(x_0)=f(x_0)-r(x_0)(1+f(x_0))=f(x_0)-1-f(x_0)=-1.$

    We get \eqref{ii} easily, since $r(x)=0$ outside of $U_{x_0}.$

    We will now show that \eqref{iii} holds.
    \begin{align*}
        \|g_+-f\|&=\sup_{x\in K}|g_+(x)-f(x)|\\&
        =\max\big\{\sup_{x\in U_{x_0}}|g_+(x)-f(x)|, \sup_{x\notin U_{x_0}}|g_+(x)-f(x)|\big\}\\&
        =\sup_{x\in U_{x_0}}|g_+(x)-f(x)|\\&
        \leq 1-(f(x_0)-\gamma)=1-f(x_0)+\gamma.
    \end{align*}

    Similarly we get that
    \begin{align*}
        \|g_--f\|&=\sup_{x\in K}|f(x)-g_-(x)|\\&
        =\max\big\{\sup_{x\in U_{x_0}}|f(x)-g_-(x)|, \sup_{x\notin U_{x_0}}|f(x)-g_-(x)|\big\}\\&
        =\sup_{x\in U_{x_0}}|f(x)-g_-(x)|\\&
        \leq (f(x_0)+\gamma)+1=1+f(x_0)+\gamma,
    \end{align*}
    hence, \eqref{iv} holds.
\end{proof}
\begin{proposition}\label{prop: G is sur isom all K}
    If $G\colon \mathbb{B}_{C(K)}\to \mathbb{B}_{C(K)}$ is such a nonexpansive mapping that for every $f\in S_{C(K)}$ and $x\in K$ with $|f(x)|=1$ we have
    $$G(f)(x)= f(x)= \pm 1,$$
    then $G$ is the identity mapping on $\mathbb{B}_{C(K)}.$
\end{proposition}
\begin{proof}
    Let $G$ be as in the statement of this proposition. Let us assume to the contrary that $G$ is not the identity mapping, i.e. $G\neq \id$.

    In other words, we obtain $f_0\in \mathbb{B}_{C(K)}$ such that for some $x_0\in K$
    \begin{equation}\label{eq: f_0 not id}
    f_0(x_0)=\id(f_0)(x_0)\neq G(f_0)(x_0).
    \end{equation}

    Now we proceed to show that under such assumptions $G$ cannot be a nonexpansive mapping, which is obviously a contradiction.

    Assume firstly, that $G(f_0)(x_0)>f_0(x_0).$ Fix $0<\gamma<\min\{G(f_0)(x_0)-f_0(x_0),|1-f_0(x_0)|,|-1-f_0(x_0)|\}.$ Note that this minimum is positive. If it were not the case, we would have $f_0(x_0)=\pm 1,$ which would contradict the hypothesis of the proposition together with (\ref{eq: f_0 not id}). 
    Now choose $g_-$ from Lemma \ref{lem: C(K) g+g-} for $x_0$ and $f_0$. Then, since
    \begin{align*}
        \|G(f_0)-G(g_-)\|&\geq |G(f_0)(x_0)-G(g_-)(x_0)|\\&
        =|G(f_0)(x_0)+1|\\&
        =1+G(f_0)(x_0),
    \end{align*}
    we get
    $$\|f_0-g_-\|\leq 1+f_0(x_0)+\gamma<1+G(f_0)(x_0)\leq \|G(f_0)-G(g_-)\|.$$

    If, however, $G(f_0)(x_0)<f_0(x_0),$ we fix $0<\gamma<\min\{f_0(x_0)-G(f_0)(x_0),|1-f_0(x_0)|,|-1-f_0(x_0)|\}$ (the minimum is positive for the same reason as in the first case) and use $g_+$ from Lemma \ref{lem: C(K) g+g-} for $x_0$ and $f_0$. Then we have
        \begin{align*}
        \|G(f_0)-G(g_+)\|&\geq |G(f_0)(x_0)-G(g_+)(x_0)|\\&
        =|G(f_0)(x_0)-1|\\&
        =1-G(f_0)(x_0)\\&
        > 1-f_0(x_0)+\gamma\\&
        \geq\|f_0-g_+\|.
    \end{align*}

    In conclusion, we see that $G$ is not a nonexpansive mapping and this is a contradiction. Hence, $G$ has to be the identity mapping.
\end{proof}
Now we present the main result regarding $C(K)$ spaces.
\begin{theorem}[Theorem~\ref{thm:collect}, (\ref{CK})]\label{thm: C(K)-result}
    Let $K$ be a compact Hausdorff topological space. Surjective isometries on $\mathbb{B}_{C(K)}$ are extremal in the space $\mc{M}$ of nonexpansive mappings $\mathbb{B}_{C(K)}\to \mathbb{B}_{C(K)}$.
\end{theorem}
\begin{proof}
    Using Lemma \ref{lemma:id_reduction}, it is sufficient to check that the identity mapping is extremal in $\mc{M}$. Assume there are nonexpansive mappings $F,G\colon \mathbb{B}_{C(K)}\to \mathbb{B}_{C(K)}$ and $\lambda\in (0,1)$ such that for $\id\colon \mathbb{B}_{C(K)}\to \mathbb{B}_{C(K)}$ we have
    $$\id=(1-\lambda) F+ \lambda G.$$
    Applying Proposition \ref{prop: F,G as id+-1} we get that $F$ and $G$ satisfy (\ref{eq: F(f)(x)=G(f)(x)=id}) for every $f\in \mathbb{B}_{C(K)}$ and $x\in K$ such that $|f(x)|=1$. This implies by Proposition \ref{prop: G is sur isom all K} that both $F$ and $G$ are the identity mapping. Hence, $\id$ is extremal. 
\end{proof}
\begin{remark}
  Let us remark a number of observations concerning this theorem:
  \begin{enumerate}[(i)]
  \item Both Lemma~\ref{lem: surj isom extr among lin nonexp} and Theorem~\ref{thm: C(K)-result} assert extremality of surjective isometries in different settings. The form of extremality given in Theorem~\ref{thm: C(K)-result} may be thought of as stronger than that of Lemma~\ref{lem: surj isom extr among lin nonexp}, since extremality in Theorem~\ref{thm: C(K)-result} is taken with respect to all nonexpansive mappings, not only the linear ones.
  \item Normally the term linear is reserved for mappings between two vector spaces. However, in the present section we also refer to linear mappings $K\to K$, where $K$ is a subset of a vector space $V$. We call a mapping $K\to K$ linear if it may be extended to a linear mapping $V\to V$.
  \item If the underlying compact space $K$ is metrisable and only has a finite number of accumulation points, by~\cite{Fakhoury} the unit ball $C(K)$ is expand-contract plastic. Hence in this case, every bijective nonexpansive self-mapping of $\mathbb{B}_{C(K)}$ is extremal.
  \end{enumerate}
\end{remark}
Regarding Theorem~\ref{thm:collect} (\ref{czero}), note that a version of Theorem \ref{thm: C(K)-result} can be stated for the unit ball of $c_0$, whereas the idea of the proof differs very little, being even simpler, since the choice of elements analogous to $g_+$ and $g_-$ provided by Lemma \ref{lem: C(K) g+g-} in the $\mathbb{B}_{C(K)}$ case, is rather obvious in the case of $\mathbb{B}_{c_0}.$ In fact, in the setting of linear mappings on $\mathbb{B}_{c_0}$ we can say even more, namely we have the following theorem.

\begin{theorem}\label{thm: T extr iff phi=1}
A linear mapping $T=(\phi_1,\phi_2,\dots)$ on $\mathbb{B}_{c_0}$ is extremal if and only if $\phi_i$ is an extreme point of $\mathbb{B}_{\ell_1}$ for every $i\in \mathbb{N}.$
\end{theorem}

As a start we present a simple characterisation of all linear mappings on $\mathbb{B}_{c_0}$ whose coordinate functionals are extreme points of the dual unit ball.
Recall that in $\mathbb{B}_{\ell_1}$ the extreme points are standard unit vectors and their opposite vectors, i.e. $\pm e_k=(0,\dots,\overbrace{\pm 1}^{k-th},0,0,\dots).$
This means that all linear mappings whose coordinate functionals are extreme points of the dual unit ball are of the form
\begin{equation}\label{eq: lin extr charac c_0}
T=(\phi_1,\phi_2,\dots)=(\varepsilon(1) e_{\pi(1)},\varepsilon(2) e_{\pi(2)},\dots),
\end{equation}
with $\pi\colon \N\to\N$ with the property that for every $k\in \pi(\N)$ the set $I_k:=\{i\in\N\colon \pi(i)=k\}$ is finite and $\varepsilon(i)=\pm 1$ for $i\in\mathbb{N}$. The condition on $I_k$ assures that our mapping has values in $c_0$.
Note that for some $i\in \N$ we may have $i\notin \pi(\N),$ i.e. not all extreme points need to be used.

\begin{example}\label{ex: mappings of form 7}
Some examples of linear mappings of the form (\ref{eq: lin extr charac c_0}) are $T=(-e_1,-e_2,-e_3,\dots)$ with $I_i=\{i\}$ for every $i\in\pi(\N)=\mathbb{N}$ and $\varepsilon=(-1,-1,\dots),$  and  $S=(e_1,e_1,e_3,e_3,e_5,e_5,\dots)$ with $I_i=\{i,i+1\}$ for every $i\in \pi(\N)=\{2n-1, n\in \mathbb{N}\}$ and $\varepsilon=(1,1,\dots).$
However, $U=(e_1,e_1,e_1,\dots)$ with $I_1=\mathbb{N}$ and $\varepsilon=(1,1,\dots)$ is not a linear mapping of the from (\ref{eq: lin extr charac c_0}), since $U((1,0,0,\dots))=(1,1,1,\dots)\notin c_0.$ 
\end{example}

\begin{remark}\label{rem: lin surj is lin extr c_0}
It is known due to Banach (see \cite{Banach}, Section~XI.5) that linear surjective isometries (called rotations in his monograph) on $\mathbb{B}_{c_0}$ are of the form
\begin{equation}\label{eq: surj isom c_0}
T(x)(i)=\varepsilon(i)x(\overline{\pi}(i))\quad \textnormal{for every}\quad i\in\N,
\end{equation}
where $(\varepsilon(i))$ is a sequence in $\{-1,1\}$ that corresponds to $T$ and $\overline{\pi}\colon \N\to \N$ is a permutation on the set $\N.$

Due to Lemma \ref{lem: Mankiewicz} and the previous remark it is easy to see that all surjective isometries on $\mathbb{B}_{c_0}$ have the form (\ref{eq: surj isom c_0}) which is a special case of the form (\ref{eq: lin extr charac c_0}). However, it is clear that not every linear mapping of the form (\ref{eq: lin extr charac c_0}) is a surjective isometry (see e.g. $S$ in  Example \ref{ex: mappings of form 7}).
\end{remark}

Before we formalise the proof of Theorem \ref{thm: T extr iff phi=1} we state two lemmas that are used in the proof. These lemmas are analogous to the lemmas used in the proof of Theorem \ref{thm: C(K)-result} for the case of $\mathbb{B}_{C(K)},$ which is why we present these results without proofs as the ideas for the proof are the same as the ones used before.
\begin{lemma}[Compare with Proposition \ref{prop: F,G as id+-1}]\label{lem: T-pi-eps, f,g}
Let $T\colon \mathbb{B}_{c_0}\to \mathbb{B}_{c_0}$ be a mapping of the form (\ref{eq: lin extr charac c_0}) with $\varepsilon$ and $\pi$. If there exist $f,g\in \mathcal{M}$ and $\lambda\in (0,1)$ such that $T=(1-\lambda) f+ \lambda g,$ then $f$ and $g$ must be such that for every $x\in S_{c_0}$ and every $k\in \mathbb{N}$ with $x(k)=1$ or $x(k)=-1,$ we have $f(x)(i)=g(x)(i)=\varepsilon(i)x({\pi(i)})=\varepsilon(i)x(k),$ for every $i\in I_k=\{i\in\N\colon \pi(i)=k\}.$ 
\end{lemma}
\begin{lemma}[Compare with Proposition \ref{prop: G is sur isom all K}]\label{lem: g is T-pi-eps}
If a nonexpansive mapping $g\colon \mathbb{B}_{c_0}\to \mathbb{B}_{c_0}$ is such that there exist $\pi\colon \mathbb{N}\to \mathbb{N}$ and a sequence of ones and minus ones $\varepsilon$ such that for every $x\in S_{\mathbb{B}_{c_0}}$, for every $k\in \mathbb{N}$ with $x(k)=\pm 1$ and for every $i\in I_k$ we have $g(x)(i)=\varepsilon(i)x_{\pi(i)}=\varepsilon(i)x(k),$ then $g$ is a linear mapping of the form (\ref{eq: lin extr charac c_0}) with exactly those $\varepsilon$ and $\pi.$
\end{lemma}
\begin{proof}[Proof of Theorem \ref{thm: T extr iff phi=1}.]
Assume firstly, that $T=(\phi_1,\phi_2,\dots)$ is a linear extremal mapping on $\mathbb{B}_{c_0}.$ We show that this implies that all $\phi_i$ are extreme points of $\mathbb{B}_{\ell_1}$, i.e. $T$ is of the form (\ref{eq: lin extr charac c_0}).

Assume to the contrary, that $T$ is not of the form (\ref{eq: lin extr charac c_0}). Firstly, let us observe the case where there is $i\in \mathbb{N}$ such that $\|\phi_i\|=\mu\in [0,1).$ Then we can define $f=(\phi_1,\phi_2,\dots,\phi_f,\phi_{i+1},\phi_{i+2},\dots)$ and $g=(\phi_1,\phi_2,\dots,\dots,0,\phi_{i+1},\phi_{i+2}),$ where $\phi_f\in S_{\ell_1}$ is such that $\phi_i=(1-\mu) \phi_f.$ Then
$$(1-\mu) f+ \mu g=T,$$
with $f$ and $g$ clearly nonexpansive. This contradicts the extremality of $T$. Consequently, $\phi_k$ has to be norm 1 for every $k\in\mathbb{N}.$

Secondly, assume all $\phi_i$ are of norm one, but there exists $i\in\mathbb{N}$ such that $\phi_i\neq \pm e_j$ for any $j\in\mathbb{N}.$ This means that $\phi_i$ has more than one non-zero coordinates. For this $i$ we define helpful coordinate functionals. Denote by $J$ the set of indices of non-zero terms of $\phi_i=(\xi_1,\xi_2,\dots).$ Obviously, $J$ can be either finite or infinite. In any case, we define for every $j\in J$
$$\phi_{i,j}=(0,0,\dots,\overbrace{\sgn(\xi_j)}^{j-th},0,0,\dots).$$
Notice that
\[\phi_i=|\xi_{j_1}|\phi_{i,{j_1}}+|\xi_{j_2}|\phi_{i,{j_2}}+\dots.\]
Taking for every $j\in J$ 
$$f_{j}=(\phi_1,\phi_2,\dots,\phi_{i-1},\phi_{i,{j}},\phi_{i+1},\phi_{i+2},\dots)$$
we get that
$$T=|\xi_{j_1}|f_{j_1}+|\xi_{j_2}|f_{j_2}+\dots.$$
Now, since $\phi_i\in S_{\ell_1}$ we know that $\sum_{i=1}^\infty |\xi_i|=\sum_{i\in J} |\xi_i|=1.$ Therefore, we may divide the set $J$ into two parts, e.g. $J=\{j_k\}\cup \bigcup_{j\in {J\backslash \{j_k\}}}\{j\}$ for some fixed $k$. For simplicity, assume $k=1$. Take $f:=\sum_{j\in {J\backslash \{j_1\}}}\frac{|\xi_j|}{1-|\xi_{j_1}|}f_j$ and $g:=f_{j_1}.$ Observe that $\norm{f}\leq 1$.
Then we can express $T$ as a nontrivial convex combination as follows
\begin{equation*}
  T=(1-|\xi_{j_1}|)\sum_{j\in {J\backslash \{j_1\}}}\frac{|\xi_j|}{1-|\xi_{j_1}|}f_j+|\xi_{j_1}| f_{j_1}=(1-|\xi_{j_1}|)f+ |\xi_{j_1}| g.   
\end{equation*}
    
Hence, $T$ is not extremal, which is a contradiction. This concludes the proof for necessity.

Assume now, that $T$ is of the form (\ref{eq: lin extr charac c_0}). We aim to show that consequently $T$ is extremal. Suppose $T$ is not extremal. Then we can find $\lambda\in(0,1)$ and distinct nonexpansive mappings $f,g$ such that $T=(1-\lambda) f+ \lambda g.$ Applying Lemma \ref{lem: T-pi-eps, f,g} and Lemma \ref{lem: g is T-pi-eps} we see that $f=g,$ which is a contradiction. Hence, $T$ is extremal. This completes the proof of sufficiency. 
\end{proof}

\begin{corollary}[Theorem~\ref{thm:collect}, (\ref{czero})]
Every surjective isometry on $\mathbb{B}_{c_0}$ is extremal.
\end{corollary}
\begin{proof}
Let $T$ be a surjective isometry on $\mathbb{B}_{c_0}.$ We know by Lemma \ref{lem: Mankiewicz} that $T$ is necessarily the restriction of a linear mapping $X\to X$. Considering Remark \ref{rem: lin surj is lin extr c_0} we know that $T$ is of the form (\ref{eq: lin extr charac c_0}). Applying Theorem \ref{thm: T extr iff phi=1} we get that $T$ is extremal. 
\end{proof}

To end this section, we point out some obvious, yet important remarks. Firstly, not all isometries on $\mathbb{B}_{c_0}$ are extremal.

\begin{example}
Consider the mappings
	\begin{equation*}
	T_{\lambda}\colon \mathbb{B}_{c_{0}}\to \mathbb{B}_{c_{0}},\qquad x\mapsto (\lambda x_{1},x_{1},x_{2},x_{3},\ldots)
	\end{equation*} 
	for $\lambda\in [0,1]$. It is easy to check that these mappings are linear isometries that are not of the form (\ref{eq: lin extr charac c_0}) for $\lambda\in [0,1)$. Notice that we have
	\begin{equation*}
	T_{1/2}=\frac{1}{2}T_{0}+\frac{1}{2}T_{1},
	\end{equation*}
	which means $T_{1/2}$ is a non-extremal linear isometry.
\end{example}

We know from Theorem \ref{thm: T extr iff phi=1} the characterisation for all linear extremal mappings on $\mathbb{B}_{c_0}$. However, dropping the condition of linearity, we still have extremal mappings.

\begin{example}
  Consider the mappings
  \begin{equation*}
    T_{\lambda}\colon \mathbb{B}_{c_{0}}\to \mathbb{B}_{c_{0}},\qquad x\mapsto (\lambda ,x_{1},x_{2},x_{3},\ldots)
  \end{equation*} 
  for $\lambda\in [-1,1]$. It is easy to check that these mappings are non-linear isometries. If $\lambda\in (-1,1),$ then $T$ is nonextremal. However, if $\lambda\in \{-1,1\},$ then $T$ is extremal. With the same idea in mind, more extremal nonlinear isometries can be constructed.
\end{example}

\section{On the extremality of the typical nonexpansive mapping}\label{sec: extr of typ nonexp map}
The goal of this section is to prove Theorem~\ref{thm:Pf}.
Let $X$ be a Banach space, $C\subseteq X$ be a non-empty, non-singleton, closed and bounded set and $\mc{M}=\mc{M}(C)$ denote the set of nonexpansive mappings $C\to C$. In the present section we view $\mc{M}$ as a complete metric space, equipped with the metric $d_{\infty}(f,g)=\sup_{x\in C}\norm{f(x)-g(x)}$. 

For $f\in\mc{M}$, consider the set
\begin{equation}\label{eq:Pf}
  P_{f}:=\set{g\in \mc{M}\colon \exists\lambda\in[0,1)\,\exists h\in \mc{M}\,\text{s.t.}f=(1-\lambda)g+\lambda h}
\end{equation}
of all those mappings $g\in\mc{M}$ which participate in a convex combination giving $f$. Note that $f\in P_{f}$ for every $f\in \mc{M}$. In Sections~\ref{sec:surj_isom} and \ref{sec:CK_c0} we have studied extremal mappings in $\mc{M}$, which may be equivalently described as those mappings $f\in\mc{M}$ with $P_{f}=\set{f}$. The objective of the present section is to investigate how close the typical mapping $f\in\mc{M}$ is to being extremal, where we understand the notion of a typical $f\in\mc{M}$ in the sense of the Baire Category Theorem. Since extremality of a mapping $f\in\mc{M}$ is equivalent to the condition $P_{f}=\set{f}$, we may interpret a mapping $f$ as being close to extremal if its set $P_{f}$ is small. We will show that a typical $f\in\mc{M}$ has the property that $P_{f}$ is a negligible subset of $\mc{M}$ in a strong sense. More precisely, we obtain the following result:
\begin{theorem}[Theorem~\ref{thm:Pf}]
  Let $X$ be normed space, $C\subseteq X$ be a non-empty, non-singleton, closed, bounded and convex set and $\mc{M}=\mc{M}(C)$ be the space of nonexpansive mappings $C\to C$ equipped with the metric $d_{\infty}(f,g)=\sup_{x\in C}\norm{f(x)-g(x)}$. For each $f\in \mc{M}$ let 
  \begin{equation*}
    P_{f}:=\set{g\in\mc{M}\colon \exists \lambda\in [0,1),\, \exists h\in \mc{M}\,\text{s.t.}\,f=(1-\lambda)g+\lambda h}
  \end{equation*}
  and let $\mc{N}$ denote the set of strict contractions $C\to C$. Then the following statements hold:
  \begin{enumerate}[(i)]
  \item\label{M-Pf_convex} For every $f\in\mc{M}$, $P_{f}$ is $F_{\sigma}$, $\mc{M}\setminus P_{f}$ is convex and there exists an affine subspace $A_{f}$ of the space of continuous mappings $C\to X$ such that $P_{f}=\mc{M}\cap A_{f}$.
  \item\label{lipf=1} For every $f\in\mc{M}\setminus \mc{N}$, i.e. every $f\in\mc{M}$ with $\lip(f)=1$, we have that $P_{f}$ is a $\sigma$-upper porous subset of $\mc{M}$. In fact $P_{f}$ is a countable union of closed, upper porous subsets of $\mc{M}$.
  \end{enumerate}
\end{theorem}

\begin{remark*}
  We remark that the set $\mc{N}$ in the above theorem is $\sigma$-lower porous, by \cite{bargetz2016sigma}. Therefore, the above theorem shows that for $f$ outside of a $\sigma$-lower porous subset of $\mc{M}$, the set $P_{f}$ is $\sigma$-upper porous and its complement $\mc{M}\setminus P_{f}$ is convex.
\end{remark*}
We will break up the proof of Theorem~\ref{thm:Pf} into several lemmas. Note that $\mc{M}$ in Theorem~\ref{thm:Pf} may be viewed as convex subset of the vector space of continuous mappings $C\to X$. Part \eqref{M-Pf_convex} of Theorem~\ref{thm:Pf} is actually something that holds for every convex subset $M$ of a vector space, whenever $P_{f}$ for $f\in M$ is defined according to \eqref{eq:Pf}. We establish this in the next three lemmas:
\begin{lemma}\label{lemma:affine_slice}
  Let $Y$ be a vector space, $M\subseteq Y$ be convex and for each $f\in M$ let
  \begin{equation*}
    P_{f}:=\set{g\in M\colon \exists\lambda\in[0,1)\,\exists h\in M\,\text{s.t.}f=(1-\lambda)g+\lambda h}.
  \end{equation*}
  Then there exists an affine subspace $A_{f}$ of $Y$ such that
  \begin{equation*}
    P_{f}=M\cap A_{f}.
  \end{equation*}
\end{lemma}
For the proof of Lemma~\ref{lemma:affine_slice} we need the following lemma:
\begin{lemma}\label{lemma:ray}
  Let $Y$, $M$, $f$ and $P_{f}$ for each $f\in M$ be given by Lemma~\ref{lemma:affine_slice}. Then the following two statements hold:
  \begin{enumerate}[(i)]
  	\item\label{ray} For each $f\in M$ and $g\in P_{f}$ we have \begin{equation*}
  	\set{f+t(g-f)\colon t\in\R}\cap M\subseteq P_{f}.
  	\end{equation*}
  	\item\label{cnvx} For each $f\in M$, $P_{f}$ is convex.
  \end{enumerate}
\end{lemma}
\begin{proof}
  We begin by verifying \eqref{ray} for a given $f\in M$ and $g\in P_{f}$. If $g=f$ then \eqref{ray} is obvious. So assume that $g\neq f$. Then, by the definition of $P_{f}$ there exist $h\in M\setminus \set{f}$ and $\lambda\in(0,1)$ such that $f=(1-\lambda)g+\lambda h$. Note that $g-f$ and $h-f$ are non-zero, linearly dependent vectors pointing in opposite directions. Hence the set in the conclusion of \eqref{ray} is unchanged if we write $h-f$ instead of $g-f$. 
  
  Let $t\in \R$ be such that $f+t(g-f)\in M$. We need to show that $f+t(g-f)\in P_{f}$. We may assume that $t>0$; otherwise interchange $g$ and $h$. Then for
  \begin{equation*}
  \mu:=\frac{\lambda t}{\lambda t+1-\lambda}\in (0,1)
  \end{equation*}
  we have that $f=(1-\mu)\br{f+t(g-f)}+\mu h$, showing that $f+t(g-f)\in P_{f}$. This completes the proof of \eqref{ray}.
  
  Next we prove \eqref{cnvx}. Suppose $f\in M$, $g_{1},g_{2}\in P_{f}$ and $\theta\in [0,1)$. Choose $\lambda_{1},\lambda_{2}\in [0,1)$ and $h_{1},h_{2}\in M$ such that
  \begin{equation*}
    f=(1-\lambda_{i})g_{i}+\lambda_{i}h_{i}\qquad \text{for }i=1,2.
  \end{equation*}
  Our aim is now to find a representation of $f$ of the form
  \begin{multline*}
    (1-\lambda)((1-\theta)g_1+\theta g_2)+\lambda((1-\mu)h_1+\mu h_2)=\\(1-\beta)((1-\lambda_{1})g_{1}+\lambda_{1}h_{1})+ \beta((1-\lambda_{2})g_{2}+\lambda_{2}h_{2}),
  \end{multline*}
  with $\lambda\in[0,1)$ and $\mu\in[0,1]$. This comes down to solving the system
  \begin{equation*}
  \begin{cases}
  (1-\lambda)(1-\theta)=(1-\beta)(1-\lambda_{1}),\\
  (1-\lambda)\theta=\beta(1-\lambda_{2}),\\
  \lambda(1-\mu)=(1-\beta)\lambda_{1},\\
  \lambda\mu=\beta\lambda_{2}.
  \end{cases}
  \end{equation*}
  of $4$ linear equations with $3$ unknowns, $\beta$, $\lambda$ and $\mu$. Solving this system, we get
  \begin{equation*}
  \beta=\frac{\theta(1-\lambda_{1})}{1-\br{\theta \lambda_{1}+(1-\theta)\lambda_{2}}},\quad \lambda=(1-\beta)\lambda_{1}+\beta\lambda_{2},\quad \mu=\frac{\beta\lambda_{2}}{(1-\beta)\lambda_{1}+\beta\lambda_{2}},
  \end{equation*}
  each of which belong to $[0,1)$. It follows that 
  \begin{equation*}
  f=(1-\lambda)\br{(1-\theta)g_{1}+\theta g_{2}}+\lambda\br{(1-\mu)h_{1}+\mu h_{2}},
  \end{equation*}
  showing that $(1-\theta)g_{1}+\theta g_{2}\in P_{f}$. Since $g_{1},g_{2}\in P_{f}$ and $\theta\in[0,1]$ were arbitrary, this proves that $P_{f}$ is convex.
\end{proof}
We are now ready to prove Lemma~\ref{lemma:affine_slice}:
\begin{proof}
  Let
  \begin{equation*}
    A_{f}:=f+\Span\set{(g-f)\colon g\in P_{f}}
  \end{equation*}
  and observe that $A_{f}$ is an affine subspace of $Y$. It is clear that $P_{f}\subseteq M\cap A_{f}$. Our task is to verify the reverse inclusion. So let $g\in M\cap A_{f}$. Then there is a minimal $n\in\N$ for which there exist $g_{1},\ldots,g_{n}\in P_{f}$ such that $g\in f+\Span\set{g_{1}-f,\ldots,g_{n}-f}$. The minimality of $n$ ensures that $g_{1}-f,\ldots,g_{n}-f$ are linearly independent vectors in $Y$. In particular, none of these vectors are zero. Therefore, for each $i=1,\ldots,n$ there exist $h_{i}\in M\setminus\set{f}$ and $\lambda_{i}\in (0,1)$ such that $f=(1-\lambda_{i})g_{i}+\lambda_{i}h_{i}$. Observe that for each $i$ the vectors $g_{i}-f$ and $h_{i}-f$ are non-zero, linearly dependent and point in opposite directions. There exist $\alpha_{1},\alpha_{2},\ldots,\alpha_{n}\in \R\setminus\set{0}$ such that
  \begin{equation*}
    g=f+\sum_{i=1}^{n}\alpha_{i}(g_{i}-f)
  \end{equation*}
  We may assume that $\alpha_{i}> 0$ for all $i$. Otherwise interchange some of the $g_{j}$'s with $h_{j}$'s. For each $i=1,\ldots,n$ let 
  \begin{equation*}
    \beta_{i}:=\frac{\alpha_{i}}{\alpha_{1}+\alpha_{2}+\ldots\alpha_{n}}\in (0,1).
  \end{equation*}
  Now consider the vector
  \begin{equation*}
    \widetilde{g}:=f+\frac{1}{n}\sum_{i=1}^{n}\beta_{i}(g_{i}-f)=\frac{1}{n}\sum_{i=1}^{n}\underbrace{\br{(1-\beta_{i})f+\beta_{i}g_{i}}}_{\in P_{f}}.
  \end{equation*}
  Note that the bracketed expression marked above belongs to $P_{f}$ because it is a convex combination of elements $f$ and $g_{i}$ of $P_{f}$, which is a convex set by Lemma~\ref{lemma:ray}\eqref{cnvx}. Thus, the expression on the right hand side is a convex combination of elements of the convex set $P_{f}$, showing that $\widetilde{g}\in P_{f}$. We conclude, using Lemma~\ref{lemma:ray}, that 
  \begin{equation*}
    \set{f+t(\widetilde{g}-f)\colon t\in \R}\cap M\subseteq P_{f}.
  \end{equation*}
  This together with the identity
  \begin{equation*}
    g=f+n(\alpha_{1}+\alpha_{2}+\ldots+\alpha_{n})(\widetilde{g}-f).
  \end{equation*}
  completes the proof.
\end{proof}

\begin{lemma}\label{lemma:M-Pf_convex}
  Let $Y$ be a Banach space, $M\subseteq Y$ be convex and for each $f\in M$ let $P_{f}$ be given by Lemma~\ref{lemma:affine_slice}. Then the following statements hold:
  \begin{enumerate}[(i)]
  \item\label{Pf_Fsigma}  The set $P_{f}$ may be expressed as the countable union
    \begin{align*}
      P_{f}&=\set{f}\cup\bigcup_{q\in\Q\cap (0,1/2)}P_{f,q}, \\ \text{where }P_{f,q}:&=\set{g\in \mc{M}\colon \exists\lambda\in[q,1-q]\,\exists h\in M\,\text{s.t.}f=(1-\lambda)g+\lambda h},\qquad q\in (0,1/2).
    \end{align*}
    Moreover, if $M$ is closed then each set $P_{f,q}$ with $q\in (0,1/2)$ is closed.
  \item\label{M-Pf_conv}  $M\setminus P_{f}$ is convex.
  \end{enumerate}
\end{lemma}
\begin{proof}
  The equation for $P_{f}$ in \eqref{Pf_Fsigma} follows immediately from the definition \eqref{eq:Pf} of $P_{f}$. The only thing in \eqref{Pf_Fsigma} that we need to check is the `Moreover' statement. So assume that $M$ is closed, let $q\in (0,1/2)$ and consider a sequence $(g_{n})_{n\in\N}$ in $P_{f,q}$ which converges in $M$ to a limit $g\in M$. We need to show that $g\in P_{f,q}$. For each $n\in\N$ let $\lambda_{n}\in [q,1-q]$ and $h_{n}\in M$ witness that $g_{n}\in P_{f,q}$. Then by passing to a subsequence of $(g_{n})_{n\in\N}$, we may assume that the limit $\lambda:=\lim_{n\to\infty}\lambda_{n}\in[q,1-q]$ exists. Using the equation $h_{n}=\frac{1}{\lambda_{n}}f-\frac{1-\lambda_{n}}{\lambda_{n}}g_{n}$ and the convergence of both sequences $(\lambda_{n})_{n\in\N}$ and $(g_{n})_{n\in\N}$ to $\lambda\in [q,1-q]$ and $g\in M$ respectively, we deduce that the limit $h:=\lim_{n\to\infty}h_{n}\in M$ exists and satisfies $f=(1-\lambda)g+\lambda h$, demonstrating that $g\in P_{f,q}$.
  
  We now turn our attention to \eqref{M-Pf_conv}. Let $g_{1},g_{2}\in M\setminus P_{f}$ and $\theta\in [0,1]$. We show that $(1-\theta)g_{1}+\theta g_{2}\in M\setminus P_{f}$. Assume the contrary. Then there exist $\lambda\in[0,1)$ and $h\in M$ such that 
  \begin{equation*}
    f=\br{1-\lambda}\br{(1-\theta)g_{1}+\theta g_{2}}+\lambda h.
  \end{equation*}
  We may rearrange the expression on the right hand side to get, for $\mu:=\lambda(1-\theta)+\theta\in [\lambda,1)$, 
  \begin{equation*}
    f=(1-\mu)g_{1}+\mu\br{\frac{\mu-\lambda}{\mu}g_{2}+\frac{\lambda}{\mu}h}.
  \end{equation*}
  Observe that $\frac{\mu-\lambda}{\mu}g_{2}+\frac{\lambda}{\mu}h$ is a convex combination of $g_{2}$ and $h$ and therefore belongs to $M$. Thus, the above representation of $f$ demonstrates that $g_{1}\in P_{f}$, which is the desired contradiction.
\end{proof} 

We are now ready to prove Theorem~\ref{thm:Pf}.
\begin{proof}[Proof of Theorem~\ref{thm:Pf}]
  Part \eqref{M-Pf_convex} is proved by applying Lemmas~\ref{lemma:affine_slice} and \ref{lemma:M-Pf_convex} with $Y$ as the space of continuous mappings $C\to X$ and $M=\mc{M}$. 
  
  We now prove part \eqref{lipf=1}. Let $f\in\mc{M}\setminus \mc{N}$. For each $q\in (0,1/2)$ let the set $P_{f,q}$ be defined as in Lemma~\ref{lemma:M-Pf_convex}\eqref{Pf_Fsigma}, by which it suffices to show that each such set is an upper porous subset of $\mc{M}$.
  
  Fix $g\in P_{f,q}$ and $\varepsilon>0$. Let $\delta=\delta(f,q,\varepsilon)>0$ be a parameter depending only on $f$, $q$ and $\varepsilon$ to be determined later in the proof. By \cite[Lemma~2.4]{bargetz2016sigma} there exist $x_{0},y\in C$ such that $\norm{x_0-y}<\varepsilon$  and 
  \begin{equation*}
    \norm{f(y)-f(x_{0})}>(1-\delta)\norm{y-x_{0}}.
  \end{equation*}
  For $\eta:=\norm{y-x_0}$ we consider the mapping $R_{\eta,x_{0}}\colon X\to X$ defined by
  \begin{equation*}
    R_{\eta,x_{0}}(x)=\begin{cases}
      x-\frac{\eta(x-x_{0})}{\norm{x-x_{0}}} & \text{ if }\norm{x-x_{0}}>\eta,\\
      x_{0} & \text{ otherwise,}
    \end{cases}
  \end{equation*}
  with $\lip(R_{\eta,x_{0}})\leq 1$ given by Medjic \cite[Lemma~4.2]{Medjic2022}. Observe also that $R_{\eta,x_{0}}(x)$ is a convex combination of $x$ and $x_{0}$ for each $x\in X$. Therefore $R_{\eta,x_{0}}|_{C}\colon C\to C$. Let $\wt{g}:=g\circ R_{\eta,x_{0}}\in\mc{M}$ and note that $\norm{\wt{g}(x)-g(x)}\leq\norm{R_{\eta,x_{0}}(x)-x} \leq \eta$ for all $x\in C$. Therefore, $d_{\infty}(\wt{g},g)\leq \eta<\varepsilon$. Moreover, it will be important that the restriction of $\wt{g}$ to the ball $x_0+\eta\mathbb{B}_X$ is constant. In particular we have that $\wt{g}(y)=\wt{g}(x_{0})=g(x_{0})$.
  
  Let $\alpha=\alpha(f,q)$ be a parameter depending only of $f$ and $q$ to be determined later in the proof, let $g'\in B_{M}(\wt{g},\alpha d_{\infty}(\wt{g},g))$ and suppose that there exist $\lambda\in[0,1]$ and $h\in\mc{M}$ such that $f=(1-\lambda)g'+\lambda h$. Then
  \begin{multline*}
    1-\delta<\frac{\norm{f(y)-f(x_{0})}}{\eta}\leq (1-\lambda)\frac{\norm{g'(y)-g'(x_{0})}}{\eta}+\lambda \frac{\norm{h(y)-h(x_{0})}}{\eta}\\
    \leq (1-\lambda)\frac{\norm{\wt{g}(y)-\wt{g}(x_{0})}+2\alpha\eta}{\eta}+\lambda=2(1-\lambda)\alpha+\lambda,
  \end{multline*}
  which rearranges to give
  \begin{equation*}
    \lambda > \frac{1-\delta-2\alpha}{1-2\alpha}=\frac{2-q}{2}>1-q,
  \end{equation*}
  where the equation above comes from setting the values of $\delta$ and $\alpha$ as
  \begin{equation*}
    \delta:=\alpha:=\frac{q}{2(1+q)}.
  \end{equation*}
  We conclude that $B_{M}(\wt{g},\alpha d_{\infty}(\wt{g},g))\cap P_{f,q}=\emptyset$, which establishes the upper porosity of~$P_{f,q}$ at~$g$.
\end{proof}

{\noindent \textbf{Acknowledgements.}} The authors would like to thank Prof.~Johann Langemets for mentioning the notion of expand-contract plastic spaces. The research of C.~Bargetz and K.~Pirk has been supported by the Austrian Science Fund (FWF): P 32523-N.

\vskip10mm
\noindent
Christian Bargetz\\
Universität Innsbruck\\
Department of Mathematics\\
Technikerstraße 13, 6020 Innsbruck, Austria\\
christian.bargetz@uibk.ac.at
\vskip5mm
\noindent
Michael Dymond\\
University of Birmingham\\
School of Mathematics, University of Birmingham, Edgbaston, Birmingham, B15 2TT, United Kingdom\\
m.dymond@bham.ac.uk\\[1mm]
Former address: Mathematisches Institut, Universit\"at Leipzig, PF 10 09 02, 04109 Leipzig, Germany
\vskip5mm
\noindent
Katriin Pirk\\
Universität Innsbruck\\
Department of Mathematics\\
Technikerstraße 13, 6020 Innsbruck, Austria\\

\end{document}